\documentclass{amsart}

\usepackage{amssymb,amsmath,amsfonts,cite}

\newtheorem*{theorem}{Theorem}
\newtheorem{corollary}{Corollary}
\newtheorem{lemma}{Lemma}

\makeatletter

\renewcommand*\subjclass[2][2010]{\def\@subjclass{#2}\@ifundefined{subjclassname@#1}{\ClassWarning{\@classname}{Unknown edition (#1) of Mathematics Subject Classification; using '2010'.}}{\@xp\let\@xp\subjclassname\csname subjclassname@#1\endcsname}}

\renewcommand{\subjclassname}{\textup{2010} Mathematics Subject Classification}

\makeatother

\begin{document}

\title[Distributive and lower-modular elements]{Distributive and lower-modular elements of the lattice of monoid varieties}

\thanks{The work is supported by the Ministry of Science and Higher Education of the Russian Federation (project FEUZ-2020-0016)}

\author{Sergey V.~Gusev}
\address{Institute of Natural Sciences and Mathematics\\
Ural Federal University\\ 620000 Ekaterinburg, Russia}
\email{sergey.gusb@gmail.com}

\date{}

\begin{abstract}
The sets of all neutral, distributive and lower-modular elements of the lattice of semigroup varieties are finite, countably infinite and uncountably infinite, respectively.
In 2018, we established that there are precisely three neutral elements of the lattice of monoid varieties.
In the present work, it is shown that the neutrality, distributivity and lower-modularity coincide in the lattice of monoid varieties.
Thus, there are precisely three distributive and lower-modular elements of this lattice.
\end{abstract}

\keywords{Monoid, variety, lattice of varieties, distributive element of a lattice, lower-modular element of a lattice.}

\subjclass{20M07, 08B15}

\maketitle

\section{Introduction and summary}
\label{Sec: introduction}

An element $x$ of a lattice $L$ is
\begin{align*}
&\text{\textit{neutral} if}&&\forall\,y,z\in L\colon\ (x\vee y)\wedge(y\vee z)\wedge(z\vee x)\\
&&&\phantom{\forall\,y,z\in L\colon}{}=(x\wedge y)\vee(y\wedge z)\vee(z\wedge x);\\
&\text{\textit{standard} if}&&\forall\,y,z\in L\colon\ (x\vee y)\wedge z=(x\wedge z)\vee(y\wedge z);\\
&\text{\textit{distributive} if}&&\forall\,y,z\in L\colon\ x\vee(y\wedge z)=(x\vee y)\wedge(x\vee z);\\
&\text{\textit{modular} if}&&\forall\,y,z\in L\colon\ y\le z\rightarrow(x\vee y)\wedge z=(x\wedge z)\vee y;\\
&\text{\textit{cancellable} if}&&\forall\,y,z\in L\colon\ x\vee y=x\vee z\ \&\ x\wedge y=x\wedge z\rightarrow y=z;\\
&\text{\textit{lower-modular} if}&&\forall\,y,z\in L\colon\ x\le y\rightarrow x\vee(y\wedge z)=y\wedge(x\vee z).
\end{align*}
\emph{Costandard, codistributive} and \emph{upper-modular} elements are defined dually to standard, distributive and lower-modular ones respectively.
Neutral, cancellable and modular elements are self-dual.
It is evident that a neutral element is both standard and costandard; a standard or costandard element is cancellable; a cancellable element is modular; a [co]distributive element is lower-modular [upper-modular]. It is well known also that a [co]standard element is [co]distributive (see~\cite[Theorem~253]{Gratzer-11}, for instance).
These types of special elements play an important role in general lattice theory; significant information about special elements in a lattice can be found in~\cite[Section~III.2]{Gratzer-11}.

Many articles were devoted to special elements of different types in the lattice $\mathbb{SEM}$ of all semigroup varieties; an overview of results published before 2015 can be found in the survey~\cite{Vernikov-15}.
In 2018, the study of special elements in the lattice $\mathbb{MON}$ of all monoid varieties  was started (referring to monoid varieties, we consider monoids as algebras of type $(2,0)$).
Now there are three articles on this topic.
In~\cite{Gusev-18}, neutral and costandard elements of the lattice $\mathbb{MON}$ are described. In~\cite{Gusev-20}, it is shown that an element of $\mathbb{MON}$ is standard if and only if it is neutral. Finally, cancellable elements of the lattice $\mathbb{MON}$ are completely determine in~\cite{Gusev-Lee-21}. 
In the present work, we continue these investigations.
We describe all distributive and lower-modular elements in the lattice $\mathbb{MON}$.

Let $\mathbf T$, $\mathbf{SL}$, and $\mathbf{MON}$ denote the variety of trivial monoids, the variety of semilattice monoids, and the variety of all monoids, respectively.
Our main result is the following theorem.

\begin{theorem}
For a monoid variety $\mathbf V$, the following are equivalent:
\begin{itemize}
\item[\textup{(i)}] $\mathbf V$ is a lower-modular element of the lattice $\mathbb{MON}$;
\item[\textup{(ii)}] $\mathbf V$ is a distributive element of the lattice $\mathbb{MON}$;
\item[\textup{(iii)}] $\mathbf V$ is a standard element of the lattice $\mathbb{MON}$;
\item[\textup{(iv)}] $\mathbf V$ is a neutral element of the lattice $\mathbb{MON}$;
\item[\textup{(v)}] $\mathbf V$ is one of the varieties $\mathbf T$, $\mathbf{SL}$ or $\mathbf{MON}$.
\end{itemize}
\end{theorem}

Note that the equivalence of items~(iv) and~(v) is proved in~\cite[Theorem~1.1]{Gusev-18}, while the equivalence of items~(iii) and~(v) is established in~\cite[Theorem~1]{Gusev-20}.

The theorem differs sharply from earlier results on special elements of the lattice $\mathbb{SEM}$. 
The set of all lower-modular elements of $\mathbb{SEM}$ is uncountably infinite (this fact easily follows from Theorem~3.2 in~\cite{Vernikov-15}); the set of all standard elements of $\mathbb{SEM}$ is countably infinite (see~\cite[Theorem~3.3]{Vernikov-15}); the set of all neutral elements of $\mathbb{SEM}$ is finite (see~\cite[Theorem~3.4]{Vernikov-15}). 
Moreover, an element of $\mathbb{SEM}$ is standard if and only if it is distributive (see~\cite[Theorem~3.3]{Vernikov-15}).
The theorem implies that these four types of special elements coincide in the lattice $\mathbb{MON}$ and the number of such elements is finite.

In general, a distributive element in a lattice need not be costandard.
The theorem together with Theorem~1.2 in~\cite{Gusev-18} implies the following interesting fact.

\begin{corollary}
\label{C: lower-modular is costandard}
Each lower-modular and so distributive element is costandard and so cancellable, codistributive, modular and upper-modular one in the lattice $\mathbb{MON}$.\qed
\end{corollary}

An element of $\mathbb{SEM}$ is modular whenever it is lower-modular (see~\cite[Corollary~3.9]{Vernikov-15}); is cancellable whenever it is distributive (compare Theorem~1.1 in~\cite{Shaprynskii-Skokov-Vernikov-19} and Theorem~3.3 in~\cite{Vernikov-15}). 
However, the costandardity does not follow from the distributivity (compare Theorems~3.3 and~3.4 in~\cite{Vernikov-15}) and the lower-modularity does not imply the cancellability in $\mathbb{SEM}$ (compare Theorem~1.1 in~\cite{Shaprynskii-Skokov-Vernikov-19} and Theorem~3.2 in~\cite{Vernikov-15}).

It is well known that the set of all neutral [standard] elements of a lattice forms a sublattice (see~\cite[Theorem~259]{Gratzer-11}).
In general, the sets of distributive or lower-modular elements in a lattice need not form a sublattice.
Nevertheless, the following fact easily follows from the theorem.

\begin{corollary}
\label{C: sublattice}
The set of all distributive [lower-modular] elements of the lattice $\mathbb{MON}$ forms a sublattice.\qed
\end{corollary}

Note that the set of all distributive elements of $\mathbb{SEM}$ forms a sublattice (this easily follows from Theorem~3.3 in~\cite{Vernikov-15}) but the set of all lower-modular elements of $\mathbb{SEM}$ does not form a sublattice (this can be easily deduced from Theorem~3.2 in~\cite{Vernikov-15}).

The article consists of three sections.
Section~\ref{Sec: preliminaries} contains definitions, notation and auxiliary results.
Section~\ref{Sec: proof of theorem} is devoted to the proof of the theorem. 

\section{Preliminaries}
\label{Sec: preliminaries}

Let~$\mathfrak X^\ast$ denote the free monoid over a countably infinite alphabet~$\mathfrak X$.
Elements of~$\mathfrak X$ are called \textit{variables} and elements of~$\mathfrak X^\ast$ are called \textit{words}.
Words unlike variables are written in bold.
An identity is written as $\mathbf u \approx \mathbf v$, where $\mathbf u,\mathbf v \in \mathfrak X^\ast$.
Let $\lambda$ denote the empty word.
We denote by $\mathrm{End}(\mathfrak X^\ast)$ the endomorphism monoid of the monoid $\mathfrak X^\ast$.
An identity is written as $\mathbf u \approx \mathbf v$, where $\mathbf u,\mathbf v \in \mathfrak X^\ast$.
For any identity system $\Sigma$, let $\mathrm{var}\,\Sigma$ denote the variety of monoids defined by $\Sigma$.

The following assertion is a specialization for monoids of a well-known universal-algebraic fact (see~\cite[Theorem~II.14.19]{Burris-Sankappanavar-81}).

\begin{lemma}
\label{L: deduction}
An identity $\mathbf u \approx \mathbf v$ holds in $\mathrm{var}\,\Sigma$ if and only if there exists some finite sequence $\mathbf u= \mathbf w_0,\mathbf w_1,\dots,\mathbf w_n= \mathbf v$ of distinct words such that for any $i\in\{0,1,\dots,n-1\}$ there exist the words $\mathbf a_i,\mathbf b_i\in \mathfrak X^\ast$, the endomorphism $\xi_i\in\mathrm{End}(\mathfrak X^\ast)$ and the identity $\mathbf s_i\approx \mathbf t_i\in\Sigma$ for which $\{\mathbf w_i,\mathbf w_{i+1}\}=\{\mathbf a_i\xi_i(\mathbf s_i)\mathbf b_i,\mathbf a_i\xi_i(\mathbf t_i)\mathbf b_i\}$.\qed
\end{lemma}

A variety of monoids is called \textit{completely regular} if it consists of \textit{completely regular} monoids, that is, unions of groups.
Let 
$$
\mathbf C = \mathrm{var}\{x^2\approx x^3,\,xy\approx yx\}.
$$
The equivalence of items a) and b) of the following lemma is well known and can be easily verified, while the equivalence of items a) and c) is established in~\cite[Corollary~2.6]{Gusev-Vernikov-18}.  

\begin{lemma}
\label{L: V is comp reg}
For a monoid variety $\mathbf V$, the following are equivalent:
\begin{itemize}
\item[\textup{a)}] $\mathbf V$ is completely regular;
\item[\textup{b)}] $\mathbf V$ satisfies the identity $x\approx x^{1+n}$ for some $n\ge 1$;
\item[\textup{c)}] $\mathbf C\nsubseteq \mathbf V$.\qed
\end{itemize}
\end{lemma}

A monoid variety is \textit{combinatorial} if all it groups are trivial.
The following statement is well known and can be easily verified.

\begin{lemma}
\label{L: V is combinatorial}
A monoid variety $\mathbf V$  is combinatorial if and only if it satisfies the identity $x^n\approx x^{n+1}$ for some $n\ge1$.\qed
\end{lemma}

Let
$$ 
\mathbf{LRB}=\mathrm{var}\{xy \approx xyx\} \ \text{ and } \ \mathbf{RRB}=\mathrm{var}\{xy \approx yxy\}.
$$
The following statement can be easily deduced from Proposition~4.7 in~\cite{Wismath-86}.

\begin{lemma}
\label{L: nsubseteq LRB}
If a variety $\mathbf V$ of idempotent monoids does not contain $\mathbf{LRB}$, then $\mathbf V$ coincides with one of the varieties $\mathbf T$, $\mathbf{SL}$ or $\mathbf{RRB}$.\qed
\end{lemma}

For any word $\mathbf w$, we denote by $\mathrm{ini}(\mathbf w)$ the word obtained from $\mathbf w$ by retaining only the first occurrence of each variable.
The following statement is well known and can be easily verified.

\begin{lemma}
\label{L: word problem LRB}
An identity $\mathbf u \approx \mathbf v$ is satisfied by $\mathbf{LRB}$ if and only if $\mathrm{ini}(\mathbf u) = \mathrm{ini}(\mathbf v)$.\qed
\end{lemma}

\section{Proof of the theorem}
\label{Sec: proof of theorem}

To prove the theorem, we need some definitions, notation and one auxiliary lemma.
The \textit{content} of a word $\mathbf w\in\mathfrak X^\ast$, that is, the set of all variables occurring in $\mathbf w$ is denoted by $\mathrm{con}(\mathbf w)$, while the length of $\mathbf w$ is denoted by $\ell(\mathbf w)$.
For a word $\mathbf w\in\mathfrak X^\ast$ and a variable $x\in\mathfrak X$, let $\mathrm{occ}_x(\mathbf w)$ denote the number of occurrences of $x$ in $\mathbf w$.
For any variety $\mathbf V$ of monoids, let $\mathrm{FIC}(\mathbf V)$ denote the fully invariant congruence on $\mathfrak X^\ast$ corresponding to $\mathbf V$.
A word $\mathbf w$ is an \textit{isoterm} for a variety $\mathbf V$ if the $\mathrm{FIC}(\mathbf V)$-class of $\mathbf w$ is singleton.
A monoid variety is \textit{proper} if it is different from $\mathbf{MON}$.

\begin{lemma}
\label{L: up LRB is not lmod}
Let $\mathbf V$ be a proper monoid variety. 
If $\mathbf{LRB}\subseteq\mathbf V$, then $\mathbf V$ is not a lower-modular element in $\mathbb{MON}$.
\end{lemma}

\begin{proof}
Since $\mathbf V$ is a proper monoid variety, $\mathbf V$ satisfies some non-trivial identity $\mathbf u_1 \approx \mathbf v_1$.
There are distinct variables $x$ and $y$ such that the identity obtained from $\mathbf u_1 \approx \mathbf v_1$ by retaining only the variables $x$ and $y$ is non-trivial.
This allows us to assume that the words $\mathbf u_1$ and $\mathbf v_1$ depend on the variables $x$ and $y$ only. 
Clearly, $\mathbf V$ satisfies the identities 
$$
x^{\mathrm{occ}_x(\mathbf u_1)+1}\approx x^{\mathrm{occ}_x(\mathbf v_1)+1} \ \text{ and } \ x^{\mathrm{occ}_y(\mathbf u_1)+1}\approx x^{\mathrm{occ}_y(\mathbf v_1)+1}
$$ 
and so the identity
$$
\mathbf u_1x^{\mathrm{occ}_x(\mathbf v_1)+1}y^{\mathrm{occ}_y(\mathbf v_1)+1}\approx \mathbf v_1x^{\mathrm{occ}_x(\mathbf u_1)+1}y^{\mathrm{occ}_y(\mathbf u_1)+1}.
$$
One can multiply both sides of the last identity on the left by the variables $x$ and $y$ and obtain the identity both sides of which contain all these variables exactly $n$ times for some $n\ge 2$. 
Therefore, $\mathbf V$ satisfies a non-trivial identity $\mathbf u \approx \mathbf v$ such that $\mathrm{con}(\mathbf u)=\mathrm{con}(\mathbf v)=\{x,y\}$ and $\mathrm{occ}_x(\mathbf u) = \mathrm{occ}_x(\mathbf v) = \mathrm{occ}_y(\mathbf u) = \mathrm{occ}_y(\mathbf v)=n$.

In view of the inclusion $\mathbf{LRB}\subseteq\mathbf V$ and Lemma~\ref{L: word problem LRB}, we may assume without any loss that $\mathrm{ini}(\mathbf u) = \mathrm{ini}(\mathbf v)=xy$.
Let $\mathbf u'$ and $\mathbf v'$ be words obtained from $\mathbf u$ and $\mathbf v$, respectively, by performing the substitution $(x,y) \mapsto (y,x)$.
Then $\mathrm{ini}(\mathbf u') = \mathrm{ini}(\mathbf v')=yx$.

Let
$$
\mathbf X = \mathrm{var}\{\mathbf ux\approx \mathbf u'x,\,\mathbf vx\approx \mathbf v'x\}\ \text{ and }\ \mathbf Y = \mathrm{var}\{\mathbf ux\approx \mathbf vx\}.
$$
Let us show that the set $\{\mathbf ux,\mathbf u'x\}$ forms a $\mathrm{FIC}(\mathbf X)$-class. 
To do this, it suffices to verify that if $\mathbf X$ satisfies $\mathbf p \approx \mathbf q$ and $\mathbf p \in \{\mathbf ux,\mathbf u'x\}$, then $\mathbf q \in \{\mathbf ux,\mathbf u'x\}$.
In view of Lemma~\ref{L: deduction}, it suffices to consider the case when $(\mathbf p,\mathbf q)=(\mathbf a \xi(\mathbf s)\mathbf b,\mathbf a \xi(\mathbf t)\mathbf b)$, where $\mathbf a,\mathbf b \in \mathfrak X^\ast$, $\xi\in\mathrm{End}(\mathfrak X^\ast)$ and either $\{\mathbf s,\mathbf t\}=\{\mathbf ux,\mathbf u'x\}$ or $\{\mathbf s,\mathbf t\}=\{\mathbf vx,\mathbf v'x\}$. 
Evidently, if $\xi(x)=\lambda$ or $\xi(y)=\lambda$, then $\xi(\mathbf s)=\xi(\mathbf t)$ and so $\mathbf q=\mathbf p \in\{\mathbf ux,\mathbf u'x\}$. 
Therefore, we may assume that the words $\xi(x)$ and $\xi(y)$ are non-empty.
Then
$$
\ell(\xi(\mathbf s))\ge\ell(\mathbf s)=2n+1=\ell(\mathbf p).
$$
Hence $\xi(x)$ and $\xi(y)$ are variables and $\mathbf a=\mathbf b=\lambda$.
Taking into account that 
$$
\mathrm{occ}_x(\mathbf p)=\mathrm{occ}_x(\mathbf s)=n+1 \ \text{ and } \ \mathrm{occ}_y(\mathbf p)=\mathrm{occ}_y(\mathbf s)=n,
$$
we have that $\xi(x)=x$ and $\xi(y)=y$.
This is only possible when $\mathbf s=\mathbf p$, whence $\{\mathbf s,\mathbf t\}=\{\mathbf ux,\mathbf u'x\}$ and so $\mathbf q \in\{\mathbf ux,\mathbf u'x\}$.

Thus, the set $\{\mathbf ux,\mathbf u'x\}$ forms a $\mathrm{FIC}(\mathbf X)$-class. 
By similar arguments we can show that $\{\mathbf vx,\mathbf v'x\}$ is a $\mathrm{FIC}(\mathbf X)$-class and $\mathbf u'x$ is an isoterm for $\mathbf Y$.
These facts, the inclusion $\mathbf{LRB}\subseteq\mathbf V$ and Lemma~\ref{L: word problem LRB} imply that the words $\mathbf ux, \mathbf u'x,\mathbf vx, \mathbf v'x$ are isoterms for $\mathbf V\vee\mathbf X$. 
Then $\mathbf u'x$ must be an isoterm for $\mathbf Y\wedge(\mathbf V\vee\mathbf X)$. 
Obviously, the variety $\mathbf Y\wedge\mathbf X$ satisfies the identity $\mathbf u'x\approx\mathbf v'x$. 
Clearly, this identity is also satisfied by the variety $\mathbf V$.
Therefore, $\mathbf V\vee(\mathbf Y\wedge\mathbf X)$ satisfies $\mathbf u'x\approx\mathbf v'x$.
Since $\mathbf V\subseteq \mathbf Y$, this implies that
$$
\mathbf V\vee(\mathbf Y\wedge\mathbf X)\subset \mathbf Y\wedge(\mathbf V\vee\mathbf X).
$$
Therefore, the variety $\mathbf V$ is not a lower-modular element of the lattice $\mathbb{MON}$, and we are done.
\end{proof}

\begin{proof}[Proof of the theorem]
As we have noted in the introduction, implications (iv)~$\Leftrightarrow$~(v) are proved in~\cite[Theorem~1.1]{Gusev-18}, implications (iii)~$\Leftrightarrow$~(v) are verified in~\cite[Theorem~1]{Gusev-20}, implication (iii)~$\Rightarrow$~(ii) follows from~\cite[Theorem~253]{Gratzer-11}, while implication (ii)~$\Rightarrow$~(i) is obvious. 
Thus, it remains to verify implication (i) $\Rightarrow$ (v).

Let $\mathbf V$ be a proper monoid variety, which is a lower-modular element in $\mathbb{MON}$. 
Then $\mathbf V$ is \textit{periodic}, i.e., it consists of periodic monoids by~\cite[Lemma~1]{Gusev-20}.
It is well known and can be easily verified that any periodic variety satisfies the identity $x^n\approx x^{n+m}$ for some $n,m\ge 1$. 
Assume that $n$ and $m$ are the least numbers such that $x^n\approx x^{n+m}$ holds in $\mathbf V$. 
Two cases are possible.

\smallskip

\textit{Case }1: $\mathbf V$ is completely regular. 
Then $n=1$ by Lemma~\ref{L: V is comp reg}. 

Suppose that $\mathbf V$ contains a non-trivial group.  
Then $m>1$ by Lemma~\ref{L: V is combinatorial}.
Put
$$
\mathbf u_1 = x^{4m+1}yx^{m+1},\ \mathbf u_2 = x^{2m+2}yx^{3m+1},\ \mathbf v_1=x^{3m+1}yx^{2m+1},\ \mathbf v_2=x^{m+2}yx^{4m+1}.
$$
Let
$$
\mathbf X = \mathrm{var}\{\mathbf u_1 \approx \mathbf u_2,\,\mathbf v_1 \approx \mathbf v_2\}\ \text{ and }\ \mathbf Y = \mathrm{var}\{\mathbf u_2 \approx \mathbf v_2\}.
$$
Let us show that the set $\{\mathbf u_1,\mathbf u_2\}$ forms a $\mathrm{FIC}(\mathbf X)$-class. 
To do this, it suffices to verify that if $\mathbf X$ satisfies $\mathbf p \approx \mathbf q$ and $\mathbf p \in \{\mathbf u_1,\mathbf u_2\}$, then $\mathbf q \in \{\mathbf u_1,\mathbf u_2\}$.
In view of Lemma~\ref{L: deduction}, it suffices to consider the case when $(\mathbf p,\mathbf q)=(\mathbf a \xi(\mathbf s)\mathbf b,\mathbf a \xi(\mathbf t)\mathbf b)$, where $\mathbf a,\mathbf b \in \mathfrak X^\ast$, $\xi\in\mathrm{End}(\mathfrak X^\ast)$ and either $\{\mathbf s,\mathbf t\}=\{\mathbf u_1,\mathbf u_2\}$ or $\{\mathbf s,\mathbf t\}=\{\mathbf v_1,\mathbf v_2\}$. 
Obviously, if $\xi(x)=\lambda$, then $\xi(\mathbf s)=\xi(\mathbf t)=\xi(y)$ and so $\mathbf q=\mathbf p \in\{\mathbf u_1,\mathbf u_2\}$. 
If $\xi(y)=\lambda$ and $\xi(x)\ne\lambda$, then $\{\xi(\mathbf s),\xi(\mathbf t)\}=\{(\xi(x))^{5m+2},(\xi(x))^{5m+3}\}$.
In either case, $(\xi(x))^{5m+2}$ is a subword of $\xi(\mathbf s)$.
But this is impossible because the words $\mathbf u_1$ and $\mathbf u_2$ do not contain any subword that is the $(5m+2)$th power of a non-empty word.
Therefore, we may further assume that the words $\xi(x)$ and $\xi(y)$ are non-empty.
Then
$$
\ell(\xi(\mathbf s))\ge\ell(\mathbf s)\ge5m+3.
$$
Hence $\xi(x)$ is a variable and, moreover, $\xi(x)=x$.
If $y \notin\mathrm{con}(\xi(y))$, then $\mathrm{con}(\xi(y))=\{x\}$ and so $(\xi(x))^{5m+2}$ is a subword of $\xi(\mathbf s)$ contradicting the above.
Thus, $y \in\mathrm{con}(\xi(y))$.
It is easy to see that this is only possible when $\xi(y)=y$ and $\mathbf p=\xi(\mathbf s)=\mathbf s$.
Hence $\mathbf q \in\{\mathbf u_1,\mathbf u_2\}$.

Therefore, $\{\mathbf u_1,\mathbf u_2\}$ forms a $\mathrm{FIC}(\mathbf X)$-class. 
By similar arguments we can show that $\{\mathbf v_1,\mathbf v_2\}$ is a $\mathrm{FIC}(\mathbf X)$-class and $\mathbf u_1$ is an isoterm for $\mathbf Y$.
Note that $\mathbf V$ violates the identities $\mathbf u_1\approx\mathbf u_2$ and $\mathbf v_1\approx\mathbf v_2$ because any variety satisfying one of these identities must satisfy the identity $x^{5m+3} \approx x^{5m+4}$ and so must be combinatorial by Lemma~\ref{L: V is combinatorial}.
In view of the above, this implies that the words $\mathbf u_1, \mathbf u_2,\mathbf v_1, \mathbf v_2$ are isoterms for $\mathbf V\vee\mathbf X$. 
Then $\mathbf u_1$ must be an isoterm for $\mathbf Y\wedge(\mathbf V\vee\mathbf X)$. 
Obviously, the variety $\mathbf Y\wedge\mathbf X$ satisfies the identity $\mathbf u_1\approx\mathbf v_1$. 
Clearly, this identity is also satisfied by the variety $\mathbf V$ because it is a consequence of the identity $x\approx x^{m+1}$.
Therefore, $\mathbf V\vee(\mathbf Y\wedge\mathbf X)$ satisfies $\mathbf u_1\approx\mathbf v_1$.
Since $\mathbf V\subseteq \mathbf Y$, it follows that
$$
\mathbf V\vee(\mathbf Y\wedge\mathbf X)\subset \mathbf Y\wedge(\mathbf V\vee\mathbf X).
$$
Thus, $\mathbf V$ is not a lower-modular element of the lattice $\mathbb{MON}$. 
This contradicts the assumption that $\mathbf V$ contains a non-trivial group.

So, it remains to consider the case when the variety $\mathbf V$ is combinatorial. 
Then $\mathbf V$ is an idempotent variety because every combinatorial completely regular variety consists of idempotent monoids.
It follows from Lemma~\ref{L: up LRB is not lmod} and the dual to it that $\mathbf{LRB},\mathbf{RRB}\nsubseteq \mathbf V$.
Then Lemma~\ref{L: nsubseteq LRB} and the dual to it imply that $\mathbf V$ coincides with one of the varieties $\mathbf{SL}$ or $\mathbf T$, and we are done.

\smallskip

\textit{Case }2: $\mathbf V$ is not completely regular. 
Then $\mathbf C\subseteq\mathbf V$ by Lemma~\ref{L: V is comp reg}.
Lemma~\ref{L: up LRB is not lmod} allows us to assume that $\mathbf{LRB}\nsubseteq \mathbf V$.

According to Lemma~\ref{L: word problem LRB}, $\mathbf V$ satisfies some identity $\mathbf u \approx \mathbf v$ with $\mathrm{ini}(\mathbf u)\ne\mathrm{ini}(\mathbf v)$. 
As in the proof of Lemma~\ref{L: up LRB is not lmod}, one can choose the identity $\mathbf u \approx \mathbf v$ so that $\mathrm{con}(\mathbf u)=\mathrm{con}(\mathbf v)=\{x,y\}$ and $\mathrm{occ}_x(\mathbf u) = \mathrm{occ}_x(\mathbf v) = \mathrm{occ}_y(\mathbf u) = \mathrm{occ}_y(\mathbf v)=k$ for some $k\ge 2$ and, moreover, $x^k$ and $y^k$ are not subwords of $\mathbf u$ and $\mathbf v$.

Let 
$$
\mathbf E=\mathrm{var}\{x^2\approx x^3,\,x^2y\approx xyx,\,x^2y^2\approx y^2x^2\}.
$$
Suppose that $\mathbf E\nsubseteq \mathbf V$.
In is verified in~\cite[Proposition~4.1]{Lee-12} that $\mathbf E\subset\mathbf C\vee\mathbf{LRB}$.
This fact and the inclusion $\mathbf C\subseteq\mathbf V$ imply that
$$
(\mathbf V\vee\mathbf E)\wedge(\mathbf V\vee\mathbf{LRB})=\mathbf V\vee\mathbf E.
$$
Evidently, $\mathbf u\approx \mathbf v$ holds in $\mathbf E$. 
It follows that $\mathbf V\vee\mathbf E$ satisfies $\mathbf u\approx \mathbf v$.
Hence $\mathbf{LRB}\nsubseteq\mathbf V\vee\mathbf E$ by Lemma~\ref{L: word problem LRB} and the fact that $\mathrm{ini}(\mathbf u)\ne\mathrm{ini}(\mathbf v)$.
This fact, the evident inclusion $\mathbf{SL}\subset\mathbf V\vee\mathbf E$ and Lemma~\ref{L: nsubseteq LRB} imply that
$$
\mathbf V\vee((\mathbf V\vee\mathbf E)\wedge\mathbf{LRB})=\mathbf V\vee\mathbf{SL}=\mathbf V,
$$
contradicting the fact that $\mathbf V$ is a lower-modular element in $\mathbb{MON}$.

Thus, it remains to consider the case when $\mathbf E\subseteq \mathbf V$.
Let
$$
\mathbf X = \mathrm{var}\{xt\mathbf u\approx tx\mathbf u,\,xt\mathbf v\approx tx\mathbf v\}\ \text{ and }\ \mathbf Y = \mathrm{var}\{tx\mathbf u\approx tx\mathbf v\}.
$$
Let us show that the set $\{xt\mathbf u,tx\mathbf u\}$ forms a $\mathrm{FIC}(\mathbf X)$-class. 
It suffices to verify that if $\mathbf X$ satisfies $\mathbf p \approx \mathbf q$ and $\mathbf p \in \{xt\mathbf u,tx\mathbf u\}$, then $\mathbf q \in \{xt\mathbf u,tx\mathbf u\}$.
In view of Lemma~\ref{L: deduction}, it suffices to consider the case when $(\mathbf p,\mathbf q)=(\mathbf a \xi(\mathbf s)\mathbf b,\mathbf a \xi(\mathbf t)\mathbf b)$, where $\mathbf a,\mathbf b \in \mathfrak X^\ast$, $\xi\in\mathrm{End}(\mathfrak X^\ast)$ and either $\{\mathbf s,\mathbf t\}=\{xt\mathbf u,tx\mathbf u\}$ or $\{\mathbf s,\mathbf t\}=\{xt\mathbf v,tx\mathbf v\}$. 
Obviously, if $\xi$ maps one of the variables $x$ or $t$ to the empty word, then $\xi(\mathbf s)=\xi(\mathbf t)$ and so $\mathbf q=\mathbf p \in\{xt\mathbf u,tx\mathbf u\}$. 
Therefore, we may assume that the words $\xi(x)$ and $\xi(t)$ are non-empty.
If $\xi(y)=\lambda$, then $\{\xi(\mathbf s),\xi(\mathbf t)\}=\{\xi(x)\xi(t)(\xi(x))^k,\xi(t)(\xi(x))^{k+1}\}$.
Since $\xi(x)\ne \lambda$, $\mathrm{occ}_y(\mathbf p)=k$ and $\mathrm{occ}_t(\mathbf p)=1$, this is only possible when $\xi(x)=x$.
However, the last equality contradicts the assumption that the word $x^k$ is not a subword of $\mathbf u$. 
Thus, $\xi(y)\ne \lambda$.
Then
$$
\ell(\xi(\mathbf s))\ge\ell(\mathbf s)=2k+2=\ell(\mathbf p).
$$
Hence $\xi(x)$, $\xi(y)$ and $\xi(t)$ are variables and $\mathbf a=\mathbf b=\lambda$.
Then $\xi(x)=x$, $\xi(y)=y$ and $\xi(t)=t$ because
$$
\mathrm{occ}_x(\mathbf p)=\mathrm{occ}_x(\mathbf s)=k+1, \ \mathrm{occ}_y(\mathbf p)=\mathrm{occ}_y(\mathbf s)=k \ \text{ and } \ \mathrm{occ}_t(\mathbf p)=\mathrm{occ}_t(\mathbf s)=1.
$$
This is only possible when $\mathbf s=\mathbf p$, whence $\{\mathbf s,\mathbf t\}=\{xt\mathbf u,tx\mathbf u\}$ and so $\mathbf q \in\{xt\mathbf u,tx\mathbf u\}$.

Thus, the set $\{xt\mathbf u,tx\mathbf u\}$ forms a $\mathrm{FIC}(\mathbf X)$-class. 
By similar arguments we can show that $\{xt\mathbf v,tx\mathbf v\}$ is a $\mathrm{FIC}(\mathbf X)$-class and $xt\mathbf u$ is an isoterm for $\mathbf Y$.
The inclusion $\mathbf E\subseteq\mathbf V$ and Theorem~4.1(i) in~\cite{Sapir-21} imply that $xt\mathbf u$ and $tx\mathbf u$ lie in different $\mathrm{FIC}(\mathbf V)$-classes.
In view of the above, this implies that these words are isoterms for the variety $\mathbf V\vee\mathbf X$. 
Then the word $xt\mathbf u$ is an isoterm for $\mathbf Y\wedge(\mathbf V\vee\mathbf X)$. 
Obviously, $\mathbf Y\wedge\mathbf X$ satisfies the identity $xt\mathbf u\approx xt\mathbf v$. 
Clearly, this identity is also holds in $\mathbf V$.
Thus, $\mathbf V\vee(\mathbf Y\wedge\mathbf X)$ satisfies $xt\mathbf u\approx xt\mathbf v$.
Since $\mathbf V\subseteq \mathbf Y$, it follows that
$$
\mathbf V\vee(\mathbf Y\wedge\mathbf X)\subset \mathbf Y\wedge(\mathbf V\vee\mathbf X).
$$
This, however, contradicts the fact that $\mathbf V$ is a lower-modular element of the lattice $\mathbb{MON}$. 
Thus, Case~2 is impossible.
\end{proof}

\end{document}